\documentclass{article}
\usepackage[letterpaper,margin=1.0in]{geometry}
\usepackage{amsthm}

\usepackage{amsmath,amsfonts,bm,amssymb}

\def\eqref#1{equation~\ref{#1}}

\def\1{\bm{1}}

\def\eps{{\epsilon}}

\def\vzero{{\bm{0}}}

\def\vd{{\bm{d}}}
\def\ve{{\bm{e}}}

\def\vn{{\bm{n}}}

\def\vu{{\bm{u}}}
\def\vv{{\bm{v}}}
\def\vw{{\bm{w}}}
\def\vx{{\bm{x}}}
\def\vy{{\bm{y}}}

\def\mF{{\bm{F}}}
\def\mG{{\bm{G}}}

\def\mP{{\bm{P}}}

\DeclareMathAlphabet{\mathsfit}{\encodingdefault}{\sfdefault}{m}{sl}
\SetMathAlphabet{\mathsfit}{bold}{\encodingdefault}{\sfdefault}{bx}{n}

\def\gN{{\mathcal{N}}}
\def\gO{{\mathcal{O}}}

\def\gT{{\mathcal{T}}}
\def\gU{{\mathcal{U}}}
\def\gV{{\mathcal{V}}}

\def\gX{{\mathcal{X}}}

\def\sR{{\mathbb{R}}}

\newtheorem{thm}{Theorem}[section]
\newtheorem{dfn}{Definition}[section]

\newtheorem{lem}{Lemma}[section]
\newtheorem{asm}{Assumption}[section]

\newtheorem{cond}{Condition}[section]

\newtheorem{prop}{Proposition}[section]

\usepackage{hyperref}
\usepackage{enumitem}
\usepackage[utf8]{inputenc} 
\usepackage[T1]{fontenc}    
\usepackage{hyperref}       
\usepackage{url}            
\usepackage{booktabs}       
\usepackage{amsfonts}       
\usepackage{nicefrac}       
\usepackage{microtype}      
\usepackage{graphicx}       
\usepackage[numbers,sort]{natbib}
\usepackage{multirow}
\usepackage{bbding}
\graphicspath{{media/}}     
\usepackage{color}
\usepackage{algpseudocode,algorithm}
\usepackage{nicefrac}
\usepackage{cancel}
\usepackage{graphicx} 
\usepackage{booktabs}
\usepackage{makecell}
\usepackage{multirow}
\usepackage{amssymb}

\DeclareMathOperator{\Proj}{Proj}
\usepackage{thmtools}
\usepackage{thm-restate}

\title{Optimal High-Order Methods for Solving \\
Monotone Variational Inequalities}

\author{Xinliang Zhang \textsuperscript{* 1} \qquad
Lesi Chen \textsuperscript{* 1} \qquad Linxuan Pan \textsuperscript{* 2} \\
Chengchang Liu \textsuperscript{3} \qquad
Junchi Yang \textsuperscript{2}
\qquad Jingzhao Zhang \textsuperscript{1} \\
\vspace{-2mm} \\
\normalsize{\textsuperscript{1} 
Institute for Interdisciplinary Information Sciences (IIIS), Tsinghua University} \\
\normalsize{\textsuperscript{2} School of Data Science, the Chinese University of Hong Kong (Shenzhen)}  \\
\normalsize{
\textsuperscript{3} Department of Artificial Intelligence, Westlake University} \\
\vspace{-2mm} \\
\normalsize{ \texttt{ \{xinliang23,chenlc23\}@mails.tsinghua.edu.cn, panlinxuan@link.cuhk.edu.cn}} \\
\normalsize{\texttt{
liuchengchang@westlake.edu.cn,
yangjunchi@cuhk.edu.cn,
jingzhaoz@mail.tsinghua.edu.cn
}}}

\begin{document}
\maketitle
\begingroup
\begin{NoHyper}
\renewcommand\thefootnote{*}
\footnotetext{Equal contributions.}
\end{NoHyper}
\endgroup
\begin{abstract}
We study second- and higher-order methods for solving smooth monotone variational inequalities (MVI). 
Monteiro and Svaiter (SIAM J. Optim., 2012) showed that a second-order method, NPE, converges at a rate of $\mathcal{O}(T^{-1.5})$. For convex-concave minimax optimization, a subclass of MVI problems, Chen, Liu, Luo, and Zhang (COLT 2025) recently improved this rate to $\tilde{\mathcal{O}}( T^{-1.75})$. However, the result has a substantial gap compared to the lower bound of $\Omega(T^{-2.5})$ established by Chen et al. (2026). In this paper, we propose a novel second-order method that achieves the optimal rate of $\gO(T^{-2.5})$. Our algorithm also extends to higher-order methods: for any integer $ p \ge 1$, we obtain a $p$th-order method with a convergence rate of $\gO(T^{-(3p-1)/2})$, matching the known lower bounds and therefore establishing optimal rates across all orders.
\end{abstract}

\section{Introduction}
Let \(F:\sR^d \to \sR^d\) be a continuously differentiable monotone operator. We study the monotone variational inequality (MVI) problem \citep{facchinei2003finite}, which aims to find a solution $\vx^* \in \gX$ such that
\begin{align} \label{eq:prob-mvi}
    \langle \mF(\vx^*), \vx - \vx^* \rangle \ge0 \quad \forall \vx \in \gX,
\end{align}
where $\gX \subseteq\mathbb{R}^{d}$ is a nonempty compact convex set.
Let $\gN_\gX(\vx)$ be the normal cone of $\gX$ at point $\vx$, equivalently, the subdifferential of the indicator function of~$\gX$. The MVI problem is equivalent to the monotone inclusion problem $\vzero \in \mF(\vx^*) + \gN_\gX(\vx^*)$, which
captures a lot of optimization problems, especially finding game-theoretical equilibria \citep{kinderlehrer2000introduction,giannessi1995variational,nocedal1999numerical,ba2025doubly,jordan2025adaptive}. A typical example is the following  convex-concave minimax optimization problem \citep{goodfellow2020generative,cesa2006prediction}:
\begin{align} \label{eq:prob-minimax}
    \min_{\vu \in \gU} \max_{\vv \in \gV} \phi(\vu, \vv).
\end{align}
This problem can be formulated as an MVI with the gradient operator $\mF(\vu,\vv) = [\nabla_\vu \phi(\vu,\vv), - \nabla_\vv \phi(\vu,\vv)]$. Minimax optimization has wide applications in machine learning, including adversarial training \citep{zhang2018mitigating}, AUC maximization~\citep{ying2016stochastic}, and distributionally robust optimization~\citep{carmon2022distributionally}. 

Following the oracle model established in classical textbooks \citep{nemirovskij1983problem,nesterov2018lectures}, we consider the $p$th-order algorithm class that can query the derivative information of $\mF$ up to order $p-1$:
\begin{equation} \label{eq:pth-order-oracle}
    (\mF(\vx), D \mF(\vx),\ldots, D^{p-1} \mF(\vx)).
\end{equation}
We also assume that the operator is $p$th-order $L_p$-smooth, \textit{i.e.}, $D^{p-1} \mF$ is $L_p$-Lipschitz continuous. 
The optimal oracle complexity for convex optimization has been settled, where the monotone operator $\mF(\vx) = \nabla f(\vx)$ is the gradient of a convex function $f$. 
For first-order methods ($p=1$), \citet{nemirovskij1983problem} showed a lower bound of $\Omega(L_1/T^2)$, and
\citet{nesterov1983method} proposed the optimal method that achieves the matching convergence rate of $\gO(L_1/T^2)$, where $T$ is the total number of oracle calls. 
For second- and higher-order methods ($p\ge2$), \citet{monteiro2013accelerated} proposed near-optimal methods that converge at a rate of $\tilde \gO(L_2/T^{3.5})$. Subsequently, \citet{kovalev2022first} and \citet{carmon2022optimal} independently improved this rate to $\gO(L_2/T^{3.5})$ for $p=2$ and $\gO(L_p/T^{(3p+1)/2})$ for general $ p\ge 2$. 
\citet{arjevani2019oracle} provided matching lower bounds of $\Omega(L_p/T^{(3p+1)/2})$ for all $p$th-order methods.

However, the optimality of the complexity of MVI beyond convex optimization remains open except for first-order methods. 
For $p=1$, \citet{korpelevich1976extragradient} proposed the extragradient method with a convergence rate of $\gO(L_1/T)$, and \citet{nemirovski2004prox} showed a matching lower bound of $\Omega(L_1/T)$ in 2004. 
For $p\geq 2$, there  are still gaps between the upper- and lower bounds.
In 2012, \citet{monteiro2012iteration} proposed the Newton Proximal Extragradient (NPE) method with a fast second-order convergence rate of $\gO(L_2/T^{1.5})$ for $p=2$, where the cost at each iteration is nearly the same as that of matrix inversion/multiplication \citep{duan2023faster,demmel2007fast}. 
The $p$th-order generalization of NPE achieves the rate of $\gO(L_p/T^{(p+1)/2})$~\citep{bullins2022higher,adil2022optimal,huang2022approximation,lin2023monotone,jiang2022generalized,lin2022perseus,nesterov2023high}, which is speculated to be optimal in \citet{adil2022optimal,lin2022perseus}.

\citet{chen2025solving} leveraged the optimal methods for convex optimization to solve the convex-concave minimax problem (\ref{eq:prob-minimax}) with convergence rates of $\tilde \gO(L_2/T^{1.75})$ and $\tilde \gO(L_p/T^{(3p+1)/4})$ using second- and $p$th-order oracles~\citep{chen2026solving}, respectively.
These are faster than the rates achieved by NPE and generalized NPE methods for the general MVI problem~(\ref{eq:prob-mvi}).
However, it is unknown whether the $\gO(L_p/T^{(p+1)/2})$ rate \citep{monteiro2012iteration,lin2022perseus} can be improved without leveraging the minimax problem structure. 
More importantly, the optimal high-order complexity of both the MVI problem~\eqref{eq:prob-mvi} and the convex-concave minimax problem~\eqref{eq:prob-minimax} remains unresolved, as the best known upper bounds still fall short of the lower bound $\Omega(L_p/T^{(3p-1)/2})$ for $p$th-order oracles~\citep{chen2026solving}.

\paragraph{Our results.} In this paper, we fully resolve the optimal high-order complexity for MVI problems. For any $p\ge 2$, our method achieves the optimal rate of $\gO(L_p/T^{(3p-1)/2})$. The approach is based on a novel accelerated Halpern method that converges at the optimal rate. Each iteration requires obtaining a resolvent operator, which can be solved efficiently at an amortized cost of $\gO(1)$ using the anchored tensor method (ATM) introduced in our previous note \citep{chen2026halpern}.

\paragraph{AI Usage.} The authors have been seeking the solution to this problem since their prior work \citep{chen2026solving}, which gives an upper bound of $\tilde \gO(L_p/T^{(3p+1)/4})$ and a lower bound of $\Omega(L_p/T^{(3p-1)/2})$ for high-order convex-concave minimax optimization. Following the release of ChatGPT 5.6 Pro Sol, they posed the problem to Hengyu Wang and Yongchao Chen at Apex Intelligence. A search using their company's system revealed that using the Halpern iteration with an anchored tensor method (ATM) as the resolvent operator's sub-solver achieves a convergence rate of $\tilde \gO(L_p/T^{p})$. This intermediate result was independently recorded in their previous short note~\citep{chen2026halpern}. More recently, upon the release of ChatGPT 6 Astra, the authors engaged in multi-round interactions with it, deriving a novel accelerated Halpern iteration as the outer loop that leads to the optimal convergence rate of $\gO(L_p/T^{(3p-1)/2})$ presented in this paper. 


We obtained our results on September 9, as documented by the publicly available ChatGPT record\footnote{https://chatgpt.com/share/6aae8bce-dc34-83ee-807b-dec17139119a}. During the preparation of this manuscript, \citet{zhang2026matching} independently established a near-optimal upper bound of $\tilde{\gO}(L_p/T^{(3p-1)/2})$ and posted their work on arXiv on September 11. In comparison, our result removes the additional logarithmic factors and tightly matches the optimal convergence rate. Our approach also leads to a simpler algorithmic framework and a more concise analysis.

\subsection{Related Works}

\paragraph{Convex optimization.} When the operator $\mF(\vx) = \nabla f(\vx)$ is the gradient of a convex function $f$, \citet{nesterov2006cubic} proposed the cubic regularized Newton (CRN) method, which is the first globally convergent second-order method and achieves a rate of $\gO(L_2/T^2)$ for $p=2$.
\citet{nesterov2008accelerating} proposed the accelerated CRN to achieve a fast rate of $\gO(L_2/T^3)$. \citet{monteiro2013accelerated} proposed the accelerated Newton proximal extragradient (A-NPE) method that converges at an even faster rate of $\tilde \gO( L_2/ T^{3.5})$. For $p \ge 2$, \citet{gasnikov2019optimal,bubeck2019near,jiang2021optimal} proposed a $p$th-order generalization of A-NPE that converges at the rate of $\tilde \gO(L_p / T^{(3p+1)/2})$. Very recently, \citet{kovalev2022first,carmon2022optimal} removed the bisection sub-procedure in A-NPE and achieved the optimal rate of $\gO(L_p / T^{(3p+1)/2})$.

Regarding lower bounds, \citet{agarwal2018lower} showed a lower bound of $\Omega(L_p/T^{(5p+1)/2})$ for randomized algorithms. Concurrently, \citet{arjevani2019oracle} showed the optimal lower bound of $\Omega( L_p/T^{(3p+1)/2})$ for deterministic algorithms. Recently, \citet{garg2021near} improved the lower bound of randomized and quantum algorithms to $\tilde \Omega( L_p/T^{(3p+1)/2})$, which matches the upper bounds up to logarithmic factors.

\paragraph{Monotone variational inequalities.} For a general operator $\mF$, \citet{monteiro2012iteration} proposed the Newton proximal extragradient (NPE) method that globally converges at the rate of $\gO(L_2/T^{1.5})$ for $p=2$, using $T$ second-order oracle calls and $\gO(T \log T)$ matrix inversion operations. \citet{bullins2022higher} generalized NPE to $p$th order and showed a convergence rate of $\gO(L_p/T^{(p+1)/2})$. Subsequently, many simpler analyses or alternative algorithms have been found \citep{huang2022approximation,adil2022optimal,lin2023monotone,jiang2022generalized,nesterov2023high,lin2022perseus}, but all the established convergence rates are $\gO(L_p/T^{(p+1)/2})$. Moreover, \citet{jiang2024adaptive,alves2023search} proposed bisection-free methods for $p=2$ that also converge at the rate of $\gO(L_2/T^{1.5})$ and only require a single matrix inversion at each iteration.

When $\mF$ is the gradient operator of a convex-concave function $\phi: \sR \rightarrow \sR$, \citet{chen2025solving} applied a primal-dual Monteiro-Svaiter acceleration \citep{monteiro2013accelerated} to the proximal function and achieved a fast rate of $\gO(L_2/T^{1.75})$ for second-order minimax optimization ($p=2$). In the full version, \citet{chen2026solving} showed that the $p$th-order generalization achieves the rate of $\gO(L_p/ T^{(3p+1)/4})$ and also established a lower bound of $\Omega(L_p/T^{(3p-1)/2})$. This paper fully closes the gap left by their work by proposing a novel method whose complexity matches their lower bound.

\section{Preliminaries}

\paragraph{Notations.}  We use $\Vert \, \cdot \, \Vert$ to denote the Euclidean norm for vectors and the spectral norm for matrices and tensors in a unified way.
We hide logarithmic factors in the notation $\tilde \gO(\,\cdot\,)$. Also, we use the notations $\gO_p(\,\cdot\,)$ and $\tilde \gO_p(\,\cdot\,)$ to hide the constants that depend on $p$. We interchangeably use $f  = \Theta(g)$ and $f \asymp g$ to denote that two functions have the same order up to constants, and also use $\Theta_p(\,\cdot\,)$ and $\asymp_p$ to hide the constants that depend on $p$.  We use $D^q \mF$ to denote the $q$th-order derivative of an operator $\mF: \gX \to \sR^d$. We also let ${\rm Proj}_{\gX}$ be the projection operator onto the set $\gX$. 

\subsection{Problem Setup}

We study the MVI Problem (\ref{eq:prob-mvi}) under the following standard assumptions \citep{huang2022approximation,lin2022perseus,bullins2022higher,jiang2022generalized}.

\begin{asm} \label{asm:X}
$\gX \subseteq\mathbb{R}^{d}$ is a nonempty compact convex set.
\end{asm}

\begin{asm} \label{asm:x-star}
 We assume there exists $\vx^{\star}$ such that $ 0 \in (\mF + \gN_\gX)(\vx^{\star})$.
\end{asm}

\begin{asm} \label{asm:F}
    \(\mF:\sR^d \to\mathbb{R}^{d}\) is continuous and monotone:
\[
\langle \mF(\vx)-\mF(\vy),\vx-\vy\rangle\geq0,\quad \forall \vx,\vy\in \gX.
\]
\end{asm}
\begin{asm} \label{asm:pth-smooth}
Assume $\mF: \sR^d \to \sR^d$ is $p$th-order $L_p$-smooth:
\begin{equation} \label{eq:F-pth-smooth}
    \left\|D^{p-1}\mF(\vx)-D^{p-1}\mF(\vy)\right\|\leq L_{p}\left\|\vx-\vy\right\|, \quad \forall \vx,\vy \in \gX.
\end{equation}
\end{asm}
Under these assumptions, we aim to find an $\eps$-solution whose tangent residual~\citep{cai2022finite,cai2023accelerated,cai2024accelerated} (which is also the norm of the operator $\mF$ when $\gX = \sR^d$) is smaller than $\eps$. 

\begin{dfn} \label{dfn:epsilon-solution}
We say $\hat \vx \in \gX$ is an $\eps$-(strong) solution to Problem (\ref{eq:prob-mvi}) if 
\[
{\rm dist}(\vzero, \mF(\hat \vx ) + \gN_\gX(\hat \vx)) \le \eps.
\]
\end{dfn}

A point that satisfies the above conditions is also referred to as a strong solution/Stampacchia variational inequality solution~\citep{hartman1966some}. This is a stronger notion than a weak solution/Minty variational inequality solution~\citep{minty1962monotone} measured by the restricted gap function \citep{nesterov2007dual} ${\rm gap}(\vy):= \sup_{\vy' \in \gX} \langle \mF(\vy), \vy - \vy' \rangle$ on a compact set, because we have ${\rm gap}(\vy) \le {\rm dist}(\vzero, \mF(\vy) + \gN_\gX(\vy)) \cdot {\rm diam}(\gX)$ by the monotonicity of $\mF$ and Cauchy–Schwarz. 



\subsection{Tensor Steps}

A basic operation to leverage the $p$th-order oracle in \eqref{eq:pth-order-oracle} is the following tensor step \citep{huang2022approximation,lin2022perseus,bullins2022higher,jiang2022generalized}, which solves the MVI problem/monotone inclusion induced by a local Taylor approximation.

\begin{dfn} \label{dfn:mvi-tensor-step}
Under Assumptions \ref{asm:F} and \ref{asm:pth-smooth}, for an input point $\vx \in \gX$ the $p$th-order tensor step with regularization parameter $M\ge L_p$ outputs $\vy = \gT_\mF^p(\vx;M)$ such that
\begin{equation} \label{eq:first-order-tensor}
    \vzero \in \bar \mF_\vx^p(\vy)+\frac{M}{p!}\left\|\vy-\vx\right\|^{p-1}(\vy-\vx)+\gN_{\gX}(\vy),
\end{equation}
where \(\bar \mF_{\vx}^p(\vy):=\sum_{k=0}^{p-1}D^{k}\mF(\vx)[\vy-\vx]^{k} / k!\) is the $(p-1)$th-order Taylor expansion for $\mF$ at the center point $\vx$.
\end{dfn}

When $p=1$, the above tensor step is exactly the (projected) gradient step that can be performed using vector addition operations; when $p=2$, it can be solved in the same spirit as the cubic regularized Newton subproblem \citep{nesterov2006cubic} using a similar binary search, whose cost is nearly the same as that of matrix multiplication/inversion \citep{duan2023faster,demmel2007fast}; in general ($p \ge 2$), the MVI problem in \eqref{eq:first-order-tensor} can be solved in polynomial time using the interior point method \citep{ralph2000superlinear,qi2002smoothing} or the cutting plane method \citep{jiang2020improved}. We also recall the following lemma, which relates the output of a tensor step to the tangent residual in the convergence criterion in Definition \ref{dfn:epsilon-solution}.

\begin{lem}\label{lem:tensor-residual-certificate}
Under Assumptions \ref{asm:F} and \ref{asm:pth-smooth}, for an input point $\vx \in \gX$ the $p$th-order tensor step with $M\ge L_p$ outputs $\vy = \gT_\mF^p(\vx;M)$ such that
\begin{equation}\label{eq:atm-certificate}
\|\mF(\vy)+\vn\|\le\frac{M+L_p}{p!}\|\vx-\vy\|^p,
\end{equation}
where $\vn$ is the following element in the normal cone at $\vy$:
\begin{equation}\label{eq:atm-normal}
\vn=-\bar\mF_{\vx}^p(\vy)-\frac{M}{p!}\|\vy-\vx\|^{p-1}(\vy-\vx)\in\gN_\gX(\vy).
\end{equation}
\end{lem}
\begin{proof}
The inclusion in \eqref{eq:atm-normal} follows directly from the tensor-step condition in Definition \ref{dfn:mvi-tensor-step}. By the Taylor remainder bound under Assumption \ref{asm:pth-smooth} and the triangle inequality,
\[
\|\mF(\vy)+\vn\|
\le \|\mF(\vy)-\bar\mF_{\vx}^p(\vy)\|+\frac{M}{p!}\|\vy-\vx\|^p
\le \frac{L_p+M}{p!}\|\vy-\vx\|^p.
\]
\end{proof}
 
\subsection{Resolvent Operators}

Following our previous note \citep{chen2026halpern}, the acceleration is based on a Halpern-type iteration, which applies to a non-expansive operator $\mP$. In the context of MVI, a natural candidate for $\mP$ is the following 
resolvent/proximal operator \citep{rockafellar1976monotone,combettes2018monotone,facchinei2003finite}. 
It is well known that $\mP_\eta$ is non-expansive  if $\mF$ is monotone \citep{ryu2016primer,rockafellar1976monotone}.

\begin{dfn} \label{dfn:resolvent}
Under Assumption \ref{asm:F}, for $\eta>0$, we can define the single-valued resolvent operator as
\[
\mP_\eta(\vx) = ({\rm Id} + \eta (\mF + \gN_\gX ))^{-1} (\vx).
\]
\end{dfn}

\begin{prop}[{\citet[Section 6]{ryu2016primer}}]\label{prop:resolvent-geom}
Under Assumption \ref{asm:F}, the operator $\mP_\eta$ is non-expansive for any $\eta>0$, 
\[
\| \mP_{\eta}(\vx)-\mP_{\eta}(\vx') \| \le \| \vx - \vx' \|, \quad \forall \vx, \vx' \in \gX.
\]
\end{prop}

\section{Main Result}

 
The main theorem below achieves the optimal complexity upper bound of $T = \gO(\epsilon^{-2/(3p-1)})$ for finding an $\eps$-solution, which is equivalent to the convergence rate $ \gO( T^{-(3p-1)/2})$ as claimed.

\begin{thm}\label{thm:main}
Under Assumptions \ref{asm:X}-\ref{asm:pth-smooth}, for any integer $p \ge 2$, there exists an algorithm (Algorithm \ref{alg:inexact-halpern-anchored}) that can return an $\epsilon$-solution $\vx \in \gX$  with a $p$th-order oracle complexity of 
\[
T = {{\mathcal{O}}\Bigg(\left(  \frac{L_{p} D^p}{\epsilon}\right)^{2/(3p-1)}\Bigg),}
\]
where $D = \left\|\vx_{0}-\vx^*\right\|$ is the distance of the initial point $\vx_0 \in \gX$ to the optimal solution $\vx^*$.
\end{thm}

The upper bound in this theorem matches the lower bound of $\Omega(\epsilon^{-2/(3p-1)})$ in \citet{chen2026solving} up to constants and directly leads to a family of optimal algorithms for any integer $p$. For instance, in the basic setting $p=2$, we can obtain a superfast Newton method that achieves the optimal complexity of $\gO(\eps^{-2/5})$, which significantly improves on the classical result of $\gO(\eps^{-2/3})$ by \citet{monteiro2012iteration}. We achieve this optimal complexity bound by a double-loop method, and we now introduce each of its components in the following.


\section{Outer Loop: Inexact Accelerated Halpern Iteration}

\begin{algorithm}[htbp]  
\caption{\textsf{Inexact-Accelerated-Halpern-Iteration}$(\mF,\vx_0, \{\eta_t\}_{t=0}^{T-1}, \{ \beta_t \}_{t=0}^{T-1}, \delta_T, T, M)$}  \label{alg:inexact-halpern-anchored}
\begin{algorithmic}[1] 
\State $\vy_0=\vx_0$, $\mG_0(\vx)=\mF(\vx)+\eta_0^{-1}(\vx-\vx_0)$
\State $(\vu_0,\vn_0)=\textsf{$R$-Adaptive-ATM}(\mG_0,\vy_0,\eta_0^{-1}, M, \delta_T)$, $\vv_0=\mF(\vu_0)+\vn_0$
\State \textbf{for} $t = 1,\cdots, T-1$ \textbf{do}
\State \quad $\vx_t = (1- \beta_t) \vu_{t-1} + \beta_t \vx_0$ 
\State \quad $\vy_t = {\rm Proj}_\gX( \vx_t - \eta_t \vv_{t-1} )$
\State \quad $\mG_t(\vx) = \mF(\vx) + \eta_t^{-1}( \vx - \vx_t)$
\State \quad $(\vu_t, \vn_t) = \textsf{$R$-Adaptive-ATM}(\mG_t, \vy_t, \eta_t^{-1}, M, \delta_T)$, $\vv_t = \mF(\vu_t) + \vn_t$
\State \textbf{end for} 
\State \textbf{return} $(\vu_{T-1}, \vv_{T-1})$
\end{algorithmic}
\end{algorithm}

We propose an inexact accelerated Halpern iteration in Algorithm \ref{alg:inexact-halpern-anchored}, which improves the convergence rate of classical Halpern iteration \citep{halpern1967fixed,lieder2020convergence,diakonikolas2020halpern,chen2026halpern} from $\gO(T^{-p})$ to $\gO(T^{-(3p-1)/2})$ for solving high-order MVI problems. 
The introduction of an auxiliary variable $\vy_t = {\rm Proj}_\gX( \vx_t - \eta_t \vv_{t-1} )$ marks a fundamental departure from the classical Halpern iteration and serves as the core mechanism for optimal acceleration.
By acting as an estimate of the resolvent operator $P_{\eta_t}(\vx_t)$ in Definition \ref{dfn:resolvent},  $\vy_t$ enables tighter error control for the inner solver, thereby allowing the outer iteration to adopt larger step sizes to attain the optimal convergence rate. Specifically, the parameters are chosen as
\begin{align} \label{eq:para-Optimal-MVI}
    q = \frac{3p-1}{2}, \quad  \eta_0 = \frac{1}{L_p D^{p-1}}, \quad \beta_t = \frac{q}{t+q}, \quad \eta_t = \eta_{t-1} \frac{t+ 2q -1}{t+q} ~~ (t \ge 1),
\end{align}
where $D = \left\|\vx_{0}-\vx^*\right\|$ is the distance of the initial point $\vx_0 \in \gX$ to the optimal solution $\vx^*$. 

As in the classical Halpern iteration \citep{halpern1967fixed,lieder2020convergence} that is applied to a non-expansive operator, we follow \citet{chen2026halpern} to use an anchored tensor method (ATM) as a subroutine to approximate the non-expansive resolvent operator $P_{\eta_t}(\vx_t)$ in Definition \ref{dfn:resolvent} at each iteration. Unlike the original ATM, which relies on the distance $R$ from the initialization to the optimal point, $R$ remains unknown in our problem setup. To address this, we adopt a modified version of ATM that is adaptive to $R$. The remaining parameter~$M$ corresponds to the regularization parameter in the tensor step, and is set as $M = \gO(L_p)$ following the literature \citep{lin2022perseus,bullins2022higher,jiang2022generalized,huang2022approximation}.
The error condition of our resolvent sub-solver is stated as follows.




\begin{cond} \label{cond:subroutine}
For $0\le t<T$, the \textsf{$R$-Adaptive-ATM} subroutine always returns $\vu_t \in \gX$ and $\vn_t \in \gN_\gX(\vu_t)$ satisfying 
\[
\| \mG_t(\vu_t) + \vn_t \| \le \frac{\delta_T}{\eta_t}, \quad \delta_T:= \frac{D}{T^{3/2}},
\]
where $D = \left\|\vx_{0}-\vx^*\right\|$ is the distance of the initial point $\vx_0 \in \gX$ to the optimal solution $\vx^*$.
\end{cond}

Under the above absolute error condition for approximating the resolvent operator, we can then obtain the following convergence rate of $\gO(T^{-(3p-1)/2})$ for our method. The key in our analysis is proving the convergence rate via the following Lyapunov function:
\begin{align} \label{eq:dfn-Pt}
    P_t =  2 B_t \langle \vv_t , \vu_t - \vx_0 \rangle + C_{t+1} \| \vv_t \|^2,
\end{align}
where the sequences $B_t$ and $C_t$ are defined in \eqref{eq:B-C}. This Lyapunov function is modified from the standard ones used in analyzing Halpern iterations \citep{diakonikolas2020halpern,cai2024variance}. By establishing a recursion for $P_t$ and $P_{t-1}$, we then arrive at the following theorem.


\begin{thm} \label{thm:outer-loop}
Under Assumptions \ref{asm:X}-\ref{asm:pth-smooth}, if Algorithm \ref{alg:inexact-halpern-anchored} uses parameters (\ref{eq:para-Optimal-MVI}) and satisfies Condition \ref{cond:subroutine}, then there exists a constant $c_p$ that depends on $p$ such that
\[
\| \vv_{T-1} \| \le \frac{c_p D}{\eta_0 ( T+1)^{\frac{3p-1}{2}}}, \quad \vv_{T-1} = \mF(\vu_{T-1}) + \vn_{T-1}.
\]
\end{thm}

\begin{proof}
Let $q = (3p-1)/2$. For the Lyapunov function in \eqref{eq:dfn-Pt}, we define the following sequences 
\begin{align} \label{eq:B-C}
B_t = \frac{\Gamma (t+q+1)}{\Gamma (q+1) \Gamma (t+1)}, \quad C_t = \frac{B_t \eta_t (1- \beta_t)}{\beta_t} = B_t \eta_t\frac{t}{q}\quad(t\ge0).
\end{align}
Using Stirling's approximation that $\Gamma(t+a) / \Gamma(t+c) \asymp t^{a-c}$, we can derive that 
\begin{align*}
\eta_t =& \eta_0 \prod_{j=1}^t \frac{j+ 2q-1}{j+q} = \eta_0 \frac{\Gamma(t+2q) \Gamma(q+1)}{\Gamma(2q) \Gamma(t+q+1)} \asymp_q \eta_0 (t+q)^{q-1}, \\    
B_t &\asymp_q (t+q)^q, \quad  C_t \asymp_q \eta_0(t+q)^{2q}\quad(t\ge1).
\end{align*}
For $\vu_t$ and $\vn_t \in \gN_\gX(\vu_t)$ that satisfy Condition \ref{cond:subroutine}, we define
\[
\vv_t := \mF(\vu_t) + \vn_t, \quad \ve_t := \vu_t - \vx_t + \eta_t \vv_t.
\]
The error condition \ref{cond:subroutine} ensures that
\begin{align} \label{eq:error-e}
    \sum_{t=0}^{T-1} \| \ve_t \| \le T \delta_T = \frac{D}{\sqrt{T}} \le D.
\end{align}
Let $\mP_\eta  = ({\rm Id} + \eta ( \mF + \gN_\gX ))^{-1}$ be the resolvent operator. Therefore, in our notation, we have $ \vu_t = \mP_{\eta_t} (\vx_t + \ve_t) $. The same nonexpansiveness holds for inputs in $\mathbb R^d$, including $\vx_t+\ve_t$. Then, using the nonexpansiveness of the resolvent operator in Proposition \ref{prop:resolvent-geom} and the fact that $\mP_{\eta_t}(\vx^*) = \vx^*$, we have \[
\| \vu_t - \vx^* \| \le \| \vx_t - \vx^* \| + \| \ve_t \| \le (1 - \beta_t) \| \vu_{t-1} - \vx^* \| + \beta_t \| \vx_0 - \vx^* \| + \| \ve_t \|\quad(t\ge1).
\]
At $t=0$, $\|\vu_0-\vx^*\|\le D+\|\ve_0\|$. Therefore, we can prove by induction that 
\[
\| \vu_t - \vx^* \| \le \| \vx_0 - \vx^* \| + \sum_{j=0}^t \| \ve_j \|.
\]
Using inequality (\ref{eq:error-e}), we have
\begin{align} \label{eq:ub-D}
\| \vu_t - \vx^* \| \le 2 D, \quad \| \vu_t - \vx_0 \| \le 3 D.    
\end{align}
At initialization ($t=0$), we have $\beta_0 = B_0 =  1$ and $C_1 = 2 \eta_0$. Therefore, it is easy to verify that \(P_0 = 2 \langle \vv_0, \ve_0 \rangle\). Then, we have
\begin{align*}
\eta_0 \| \vv_0 \|
=& \| \ve_0 - (\mP_{\eta_0}(\vx_0 + \ve_0) - \mP_{\eta_0}(\vx_0)) + (\vx_0 -   \mP_{\eta_0}(\vx_0)) \|  \\
\le& (\| \ve_0 \| + \|\mP_{\eta_0}(\vx_0 + \ve_0) - \mP_{\eta_0}(\vx_0)\| + \|\vx_0 -   \mP_{\eta_0}(\vx_0) \|)\\
\le& ( 2\| \ve_0 \| + \| \vx_0 - \vx^* \| ) \le  3 D,
\end{align*}
where the first inequality is due to the triangle inequality, and the second one uses the non-expansiveness of the resolvent operator in Proposition \ref{prop:resolvent-geom}, and 
\( \|\vx_0 -   \mP_{\eta_0}(\vx_0) \| \le \| \vx_0 - \vx^* \| \) due to \citep[Lemma 5.2]{chen2026halpern}. Therefore, we have
\begin{align} \label{eq:ub-P0}
    P_0 = 2 \langle \vv_0, \ve_0 \rangle \le 6 D^2 / \eta_0.
\end{align}
For $t \ge 1$, we define the following differences of vectors as
\[
\vw_t = \vv_t - \vv_{t-1}, \quad 
\vd_t = \vu_t - \vu_{t-1} =  - \beta_t (\vu_{t-1} - \vx_0) - \eta_t \vv_t + \ve_t.
\]
Then, the monotonicity of $\mF+\gN_\gX$ implies that $ \langle \vw_t , \vd_t \rangle \ge 0$. Therefore, 
\begin{align*}
    P_t - P_{t-1} =&  -C_t \| \vw_t \|^2 - \frac{2B_{t-1}}{\beta_t} \langle \vw_t, \vd_t \rangle + \frac{2 C_t}{\eta_t} \langle \vw_t, \ve_t \rangle + 2 B_t \langle \vv_t, \ve_t \rangle \\
\le& - \frac{1}{2} C_t\| \vw_t \|^2 + \frac{2 C_{t}}{\eta_t^2} \| \ve_t \|^2 + 2 B_t \| \vv_t \| \| \ve_t \| \\
=& - \frac{1}{2} C_t\| \vw_t \|^2 + \frac{2 C_{t}}{\eta_t^2} \| \ve_t \|^2  + 2 \frac{B_t}{\sqrt{C_{t+1}}} \cdot \sqrt{C_{t+1}} \| \vv_t \| \| \ve_t \|,
\end{align*}
where the inequality uses Young's inequality that
\[
\frac{2C_t}{\eta_t} \left | \langle \vw_t, \ve_t \rangle  \right| \le \frac{1}{2} C_t \| \vw_t \|^2 + \frac{2C_t}{\eta_t^2} \| \ve_t \|^2.
\]
Define
\[
Q = \sup_{0 \le t < T} \frac{B_t}{\sqrt{C_{t+1}}}, \quad V = \sup_{0 \le t < T} \sqrt{C_{t+1}} \| \vv_t \|.
\]
Then, telescoping the above recursive inequality of $P_t - P_{t-1}$ gives
\[
P_t \le P_0 + 2 \sum_{j=1}^{t} \frac{C_j}{\eta_j^2} \| \ve_j \|^2 + 2 Q V \sum_{j=1}^{t} \| \ve_j \|.
\]
The above inequality gives an upper bound on $P_t$. For $t = \arg \max_{0 \le t < T} \sqrt{C_{t+1}} \| \vv_t\|$, that is, the index attaining $V$,  we can also combine Eqs. (\ref{eq:ub-D}) and (\ref{eq:dfn-Pt}) to obtain
\[
P_t \ge V^2 - 6 Q D V
\]
The above two inequalities, in conjunction with inequality (\ref{eq:error-e}), imply 
\begin{align*}
    V^2 \le 8 Q DV + P_0 + 2 \sum_{t=0}^{T-1} \frac{C_t}{\eta_t^2} \| \ve_t \|^2.
\end{align*}
Recalling that $C_t \asymp_q \eta_0 (t+q)^{2q} $ and  \( \eta_t \asymp_q \eta_0 (t+q)^{q-1} \), we have $C_t  / \eta_t^2 \asymp_q (t+q)^2/ \eta_0$. Since $\|\ve_t\|\le D/T^{3/2}$ and $t+q\le q(t+1)$, we also have
\begin{align} \label{eq:ub-C-eta-e}
\sum_{t=1}^{T-1}(t+q)^2\|\ve_t\|^2\le q^2D^2,\qquad
\sum_{t=1}^{T-1}\frac{C_t}{\eta_t^2}\|\ve_t\|^2= \gO_q \left({D^2}/{\eta_0} \right).    
\end{align}
Thus, there exists a constant $c_q$ that depends on $q$ such that
\[
V^2 \le c_q D^2 / \eta_0 + 8 QDV.
\]
Solving this quadratic inequality gives 
\[
V \le 8 QD + D \sqrt{{c_q}/{\eta_0}}.
\]
Recalling that $Q = \sup_{0 \le t < T}{B_t}/{\sqrt{C_{t+1}}}$ and $B_t \asymp_q (t+q)^q$, we have $Q = \gO_q(1 / \sqrt{\eta_0})$. Therefore, we can conclude that there exists a constant $c_q'$ such that
\begin{align} \label{eq:ub-V}
V \le c_q' D /\sqrt{\eta_0}.    
\end{align}
This completes the proof by noting the definitions of $q = (3p-1)/2$, $V = \sup_{0 \le t<T} \sqrt{C_{t+1}} \| \vv_t\| $, and $C_t \asymp_q  \eta_0 (t+q)^{2q}$.
\end{proof}

\section{Inner Loop: Anchored Tensor Method}

In this section, we introduce a subroutine that can solve the resolvent operator that achieves Condition \ref{cond:subroutine} required by Algorithm \ref{alg:inexact-halpern-anchored}. Recall that the resolvent in Definition \ref{dfn:resolvent} is \(\mP_\eta(\vx) = ({\rm Id} + \eta (\mF + \gN_\gX ))^{-1} (\vx)\). Equivalently, the proximal point $\vy=\mP_\eta(\vx)$ is the unique solution of
\begin{equation} \label{eq:optimality-resolvent}
    \vzero \in \mF(\vy)+ \frac{1}{\eta} (\vy-\vx)+\gN_{\gX}(\vy).
\end{equation}
For a fixed $\vx$, the above problem is equivalent to solving the MVI problem induced by the regularized operator $\mG(\vy) := \mF(\vy) + (\vy  -\vx) / \eta$, which is $\mu $-strongly monotone for $\mu = 1/\eta$:
\begin{equation} \label{eq:regularized-strongly-monotone}
    \langle \mG(\vy) -  \mG(\vy'), \vy - \vy' \rangle \ge \mu \|\vy - \vy' \|^2, \qquad \forall \vy,\vy' \in \gX.
\end{equation}
Therefore, we only need to consider solving the MVI problem induced by a strongly monotone operator.

\subsection{Control of Total Distances}
When solving the MVI induced by the operator $\mG_t(\vx) = \mF(\vx) + \eta_t^{-1}( \vx - \vx_t)$ in Algorithm \ref{alg:inexact-halpern-anchored}, our sub-solver is initialized at $\vy_t$. Hence, the complexity of solving the resolvent operator depends on $r_t=\|\vy_t-\mP_{\eta_t}(\vx_t)\|$, the distance of $\vy_t$ to the proximal point $\mP_{\eta_t}(\vx_t)$. To control the total complexity of solving the subproblems, we need the following technical lemma that bounds the total initial distances, which is also a consequence of the Lyapunov analyses in the proof of Theorem \ref{thm:outer-loop}.

\begin{lem}\label{lem:predictor-energy}
    Under the conditions of Theorem \ref{thm:outer-loop}, let $r_t=\|\vy_t-\mP_{\eta_t}(\vx_t)\|$. There exists a constant $c_p'$ that depends on $p$ such that
    \[
    \sum_{t=1}^{T-1} (t+q)^2 r_t^2 \le c_p' D^2,
    \]
where $D = \left\|\vx_{0}-\vx^*\right\|$ is the distance of the initial point $\vx_0 \in \gX$ to the optimal solution $\vx^*$.
\end{lem}

\begin{proof}
Recalling the definition of $P_t$ in Eq. (\ref{eq:dfn-Pt}), we have
\begin{align*}
    P_t =&  2 B_t \langle \vv_t , \vu_t - \vx_0 \rangle + C_{t+1} \| \vv_t \|^2 \\
    =& C_{t+1} \left \| \vv_t + \frac{B_t}{C_{t+1}} (\vu_t - \vx_0) \right\|^2 - \frac{B_t^2}{C_{t+1}} \| \vu_t - \vx_0 \|^2 \ge - \frac{B_t^2}{C_{t+1}} \| \vu_t - \vx_0 \|^2.
\end{align*}
Therefore, using \( \| \vu_t - \vx_0 \| \le 3 D \) from inequality (\ref{eq:ub-D}) and the definition $Q = \sup_{0 \le t < T}{B_t}/{\sqrt{C_{t+1}}}$,  we can obtain that $P_{T-1}\ge-9Q^2D^2$. Keeping the negative term in the recursive inequality for $P_t-P_{t-1}$ in the proof of Theorem \ref{thm:outer-loop}, we obtain
\begin{align} \label{eq:ub-12Cw}
\frac12\sum_{t=1}^{T-1}C_t\|\vw_t\|^2
&\le P_0-P_{T-1}+2\sum_{t=1}^{T-1}\frac{C_t}{\eta_t^2}\|\ve_t\|^2
+2QV\sum_{t=1}^{T-1}\|\ve_t\| = \gO_q \left( D^2/\eta_0 \right),
\end{align}
where we used the upper bound on $P_0$ in inequality (\ref{eq:ub-P0}), $Q = \gO_q(1 / \sqrt{\eta_0})$, the upper bound on $V$ in inequality~(\ref{eq:ub-V}), the upper bound on $\sum_{t=1}^{T-1} \| \ve_t\|$ in inequality~(\ref{eq:error-e}), and that of $\sum_{t=1}^{T-1} (C_t / \eta_t^2) \| \ve_t \|^2 $ in inequality~(\ref{eq:ub-C-eta-e}).
For $t\ge1$, non-expansiveness of the resolvent gives
$\|\vu_t-\mP_{\eta_t}(\vx_t)\|\le\|\ve_t\|$.
Since $\vu_t=\Proj_\gX(\vu_t)$, non-expansiveness of projection also gives
\[
\|\vu_t-\vy_t\|\le\|\vu_t-\vx_t+\eta_t\vv_{t-1}\|
=\|\ve_t-\eta_t\vw_t\|.
\]
Thus $r_t\le\eta_t\|\vw_t\|+2\|\ve_t\|$, and consequently
\begin{align*}
\sum_{t=1}^{T-1}(t+q)^2r_t^2 \le2\sum_{t=1}^{T-1}(t+q)^2\eta_t^2\|\vw_t\|^2
+8\sum_{t=1}^{T-1}(t+q)^2\|\ve_t\|^2.
\end{align*}
Since $(t+q)^2\eta_t^2 = \gO_q (\eta_0C_t)$ for $t\ge1$, we can combine inequalities (\ref{eq:ub-C-eta-e}) and (\ref{eq:ub-12Cw}) to obtain that
\[
\sum_{t=1}^{T-1}(t+q)^2r_t^2 \le c_q\eta_0\sum_{t=1}^{T-1}C_t\|\vw_t\|^2+8q^2D^2
\le c_qD^2,
\]
where $c_q$ is a constant that only depends on $q$.
\end{proof}

The above technical lemma almost implies that each resolvent operator can be solved in $\gO(1)$ amortized costs. By the existing guarantee from the ATM sub-solver by \citet[Theorem 5.1]{chen2026halpern}, each resolvent operator can be efficiently solved in the costs of roughly $\gO( r_t/ \rho_t)$, where $\rho_t \asymp_p (\eta_t L_p)^{-1/(p-1)} \asymp_p D (t+1)^{3/2}$ is linear convergence region of tensor step by \citet[Lemma 5.1]{chen2026halpern}. Then we can deduce that the total costs of solving the resolvents are $\gO(T)$ since, for any integer $p$, 
\[
\sum_{t=1}^{T-1}\frac{r_t}{\rho_t} = \gO \left(\frac{1}{D}\sum_{t=1}^{T-1}(t+q)^{3/2}r_t \right)= \gO \left(\frac{1}{D}\left(\sum_{t=1}^{T-1}(t+q)^2r_t^2\right)^{1/2}
\left(\sum_{t=1}^{T-1}(t+q)\right)^{1/2} \right) = \gO(T).
\]
See also \eqref{eq:rt-rhot} below. Now, let us formally recall the ATM sub-solver and its theoretical guarantees.

\subsection{The Anchored Tensor Method}

Our subroutine to solve the resolvent operator is a simple modification of the Anchored Tensor Method (ATM) of \citet[Algorithm 3]{chen2026halpern}. We recall their method as follows. In this method, $M = \gO(L_p)$ is the regularization parameter in the tensor step, $R>0$ is an upper bound of the initial distance, and the value of $S$ will specified later in the next subsection. 

\begin{algorithm}[htbp] 
\caption{\textsf{Anchored-Tensor-Method}$ (\mG,\vy_0, \mu, M, R, S)$}  \label{alg:anchoring}
\begin{algorithmic}[1] 
\State  Set $K=\min\{k\ge0:r_k=\rho(\mu)\}$ using \eqref{eq:rj-muj} 
\State \textbf{for} $k = 0,\cdots, K-1$ \hfill \Comment{Phase I: Enter the local region.}
\State \quad Select regularization $\mu_k$ according to \eqref{eq:rj-muj}.
\State \quad Define the regularized operator $\mG_k(\vy) = \mG(\vy) + (\mu_k - \mu) (\vy - \vy_0)$
\State \quad Perform a tensor step $\vy_{k+1} = \gT_{\mG_k}^p (\vy_k;M)$
\State \textbf{end for}
\State \textbf{for} $s = K,\cdots, K+S-1$ \hfill \Comment{Phase II: Linear convergence.}
\State \quad Perform a tensor step $\vy_{s+1} = \gT_{\mG}^p (\vy_s;M)$
\State \textbf{end for}
\State \textbf{return} $\vy_{K+S}$
\end{algorithmic}
\end{algorithm}

The starting point of the ATM subroutine is the following fact, which essentially follows from \citet[Theorem 3.5]{lin2022perseus}, which suggests that the tensor steps in Definition \ref{dfn:mvi-tensor-step} enjoy a linear convergence guarantee in the local region with radius $\rho(\mu)=a_p(\mu/L_p)^{1/(p-1)}$, where $a_p$ is a constant that depends on $p$.

\begin{lem}[{\citet[Lemma 5.1]{chen2026halpern}}] \label{lem:local-region}
For any integer $p \ge 2$, let $\mG$ be $\mu$-strongly monotone and $p$th-order $L_p$-smooth, with $\vzero\in\mG(\vy^*)+\gN_\gX(\vy^*)$.
    There exist $p$-dependent constants $a_p>0$, $m_p\ge1$, and $0<\theta_p<1$ such that, for $M=m_pL_p$, we have
    \begin{equation}\label{eq:atm-local}
    \|\vy-\vy^*\|\le\rho(\mu) \triangleq a_p(\mu/L_p)^{1/(p-1)} \quad\Longrightarrow\quad
    \|\gT_{\mG}^p(\vy;M)-\vy^*\|\le\theta_p\|\vy-\vy^*\|.
\end{equation}
\end{lem}

Following \citet{chen2026halpern}, the ATM subroutine in Algorithm \ref{alg:anchoring} consists of two phases. In the first phase with $K$ iterations, the algorithm iteratively performs the tensor step on the anchored/regularized operator \(\mG_k(\vy) := \mG(\vy) + (\mu_k-\mu)(\vy - \vy_0) \). The anchoring coefficient $\mu_k$ is chosen such that every tensor step lies in the local superlinear convergence region and $\mu_k$ decreases over iterations as the algorithm approaches the minimizer. Formally, the parameters in phase one are set as follows:
\begin{equation} \label{eq:rj-muj}
    r_k := \max \left\{ \rho(\mu), \left( \frac{1}{R} + \frac{k (1 - \theta_p)}{2R (p-1)} \right)^{-1}  \right\}, \qquad
 \mu_k := \rho^{-1}(r_k) =  \mu \left( \frac{r_k}{\rho(\mu)} \right)^{p-1},
\end{equation}
where $R>0$ satisfies $\| \vy_0 - \vy^*  \| \le R$. By this definition, $r_k$ is decreasing from $r_0$ that satisfies $r_0 \le R$  to $\rho(\mu)$. Consequently, $\mu_k$ is decreasing from $\mu_0$ such that $\mu_0 \ge \rho^{-1}(R)$ to $\mu$. It can also be shown that the minimal number of iterations $K$ such that $r_k =\rho$ satisfies $K = \gO\left( R (L_p/\mu)^{1/(p-1)} \right)$.


\begin{lem}[{\citet[Theorem 5.1]{chen2026halpern}}] \label{lem:K-enter-local}
For any integer $p \ge 2$, let
$\mG$ be $\mu$-strongly monotone and $p$th-order $L_p$-smooth, with $\vzero\in\mG(\vy^*)+\gN_\gX(\vy^*)$. Let $K=\min\{k\ge0:r_k=\rho(\mu)\}$. If $R\ge\max\{\rho(\mu),\|\vy_0-\vy^*\|\}$, then Phase I of Algorithm \ref{alg:anchoring} returns $\|\vy_{K}-\vy^*\|\le\rho(\mu)$ in $K = \gO( R/\rho(\mu) )$ iterations.
\end{lem}

Therefore, at the beginning of the second phase, the anchoring coefficient $\mu_K$ has decreased to $\mu_K = \mu$ and the algorithm has entered the local region for the unregularized operator $\mG$ in Lemma \ref{lem:local-region}. Therefore, the second phase reaches a $\delta$-solution such that $\|\vy_{K+S} - \vy^* \| \le \delta$ in an additional $S = \gO( \log (\rho(\mu) / \delta))$ iterations.

\subsection{The Adaptive Anchored Tensor Method without Initial Distance Knowledge}


The ATM subroutine introduced above can solve the resolvent operator up to high accuracy. However, this algorithm requires knowledge of the initial distance $R>0$ such that $\| \vy_0 - \vy^* \| \le R$ to set the parameters $r_k$ and $\mu_k$ in \eqref{eq:rj-muj}. However, when calling the subroutine in Algorithm \ref{alg:inexact-halpern-anchored}, this information is unavailable in advance. To address this issue, we propose the \textsf{$R$-Adaptive-ATM} subroutine in Algorithm \ref{alg:atm-absolute}, which adapts to the unknown distance by dynamically doubling the guess $R_j$ at each attempt.

\begin{algorithm}[htbp]
\caption{\textsf{$R$-Adaptive-ATM} $(\mG,\vy,\mu,M, \delta)$}\label{alg:atm-absolute}
\begin{algorithmic}[1]
\State Set $S = \gO_p(1+\log_+(\rho(\mu)/\delta))$ using \eqref{eq:atm-refinement-count}.
\For{$j=0,1,2,\ldots$}
\State Set $R_j=2^jS\rho(\mu)$.
\State $\vw=\textsf{Anchored-Tensor-Method}(\mG,\vy,\mu,M,R_j,S)$ 
\State Let $\vu=\gT_{\mG}^p(\vw;M)$ and $\vn = - \bar \mG_{\vw}^p(\vu) - M  \| \vu - \vw \|^{p-1} (\vu - \vw) /p! $ by \eqref{eq:atm-normal}.
\If{$\mu^{-1}\|\mG(\vu)+\vn\|\le\delta$}
\State \Return $(\vu,\vn)$
\EndIf
\EndFor
\end{algorithmic}
\end{algorithm}

In this algorithm, we set
\begin{equation}\label{eq:atm-refinement-count}
h_p:= \frac{2^p(m_p+1)a_p^{p-1}}{p!}, \quad
S=\max\left\{1,\left\lceil\frac{\log_+(h_p\rho(\mu)/\delta)}{p\log(1/\theta_p)}\right\rceil\right\},~~~\text{and}~~~\log_+(s):=\max\{0,\log s\},
\end{equation}
where $\delta>0$ is the target precision. The algorithm terminates when the tangent residual in Condition \ref{cond:subroutine} is sufficiently small, as tested using a final tensor step in light of Lemma \ref{lem:tensor-residual-certificate}. The guarantee of this algorithm is stated in the following lemma.

\begin{lem}\label{lem:atm-absolute}
For any integer $p \ge 2$, let $\mG$ be $\mu$-strongly monotone and $p$th-order $L_p$-smooth, with $\vzero\in\mG(\vy^*)+\gN_\gX(\vy^*)$. For $\vy\in\gX$ and $\delta>0$, Algorithm \ref{alg:atm-absolute} with $M=m_pL_p$ and $S$ given by \eqref{eq:atm-refinement-count} returns $\vu\in\gX$ and $\vn\in\gN_\gX(\vu)$ satisfying
\[
\mu^{-1}\|\mG(\vu)+\vn\|\le\delta
\]
in $\gO(1+r/\rho(\mu)+\log_+(\rho(\mu)/\delta))$ oracle calls, where $r=\|\vy-\vy^*\|$.
\end{lem}
\begin{proof}
Let $r=\|\vy-\vy^*\|$ and consider a trial with $R_j\ge r$. Let $K_j$ be the number of first-phase steps in the call to Algorithm \ref{alg:anchoring}, whose iterates start at $\vy_0=\vy$. Lemma \ref{lem:K-enter-local} gives $\|\vy_{K_j}-\vy^*\|\le\rho(\mu)$. After the $S$ second-phase steps, the returned point $\vw=\vy_{K_j+S}$ satisfies, by \eqref{eq:atm-local},
\[
\|\vw-\vy^*\|\le\theta_p^S\rho(\mu).
\]
Algorithm \ref{alg:atm-absolute} then takes the final step $\vu=\gT_{\mG}^p(\vw;M)$ and forms $\vn\in\gN_\gX(\vu)$ by \eqref{eq:atm-normal}. Applying \eqref{eq:atm-local} once more gives
\[
\|\vu-\vw\|\le\|\vu-\vy^*\|+\|\vw-\vy^*\|
\le(1+\theta_p)\theta_p^S\rho(\mu)
\le2\theta_p^S\rho(\mu).
\]
Using \eqref{eq:atm-certificate}, $M=m_pL_p$, and $\rho(\mu)^{p-1}=a_p^{p-1}\mu/L_p$, we obtain
\[
\mu^{-1}\|\mG(\vu)+\vn\|
\le\frac{(m_p+1)L_p}{\mu p!}\bigl(2\theta_p^S\rho(\mu)\bigr)^p
=h_p\theta_p^{pS}\rho(\mu)\le\delta,
\]
where the last inequality follows from \eqref{eq:atm-refinement-count}. Note that acceptance occurs no later than $j^*:=\min\{j\ge0:2^jS\rho(\mu)\ge r\}$. Each trial starts again from $\vy$ and costs $\gO(R_j/\rho(\mu)+S)=\gO(2^jS)$ oracle calls. Consequently, the total complexity is
\[
\sum_{j=0}^{j^*}\gO(2^jS)
=\gO\bigl(S+r/\rho(\mu)\bigr)
=\gO\bigl(1+r/\rho(\mu)+\log_+(\rho(\mu)/\delta)\bigr).
\]
\end{proof}


Now, we are ready to give the formal proof of our main theorem.
\begin{thm}[Full version of Theorem \ref{thm:main}]\label{thm:total-atm}
Under the assumptions and parameter choices of Theorem \ref{thm:outer-loop}, Algorithm \ref{alg:inexact-halpern-anchored} uses $\gO(T)$ total $p$th-order oracle calls and returns
\[
\|\vv_{T-1}\|\le\frac{c_pL_pD^p}{T^{(3p-1)/2}}, \qquad \vv_{T-1}\in(\mF+\gN_\gX)(\vu_{T-1}).
\]
Consequently, the $p$th-order oracle complexity for finding an $\eps$-solution is 
\[
\gO\left(\left(\frac{L_pD^p}{\epsilon}\right)^{2/(3p-1)}\right).
\]
\end{thm}
\begin{proof}
The local radius for the $t$th subproblem according to \eqref{eq:atm-local} satisfies
\[
\rho_t=a_p(\eta_tL_p)^{-1/(p-1)}\asymp_p\frac{D}{(t+1)^{3/2}},
\]
since $\eta_t\asymp_p\eta_0(t+1)^{q-1}$ and $(q-1)/(p-1)=3/2$. Lemma \ref{lem:atm-absolute} bounds the cost of solving each subproblem by  
\begin{align} \label{eq:one-two}
\gO(1+ \underbrace{r_t/\rho_t}_{\rm (I)}+ \underbrace{\log_+(\rho_t/\delta_T)}_{\rm (II)}).    
\end{align}
The total cost is then the summation of \eqref{eq:one-two} over $t = 0,\cdots, T-1$. 
By the non-expansiveness in Proposition \ref{prop:resolvent-geom}, the triangle inequality, and $\mP_{\eta_0}(\vx^*)=\vx^*$, we have
\[
r_0=\|\vx_0-\mP_{\eta_0}(\vx_0)\|
\le\|\vx_0-\vx^*\|+\|\mP_{\eta_0}(\vx^*)-\mP_{\eta_0}(\vx_0)\|
\le2D,
\]
hence $r_0/\rho_0=\gO(1)$ for any constant $p$. For the remaining costs contributed by ${\rm (I)}$, Lemma \ref{lem:predictor-energy} and Cauchy--Schwarz indicate that, for any constant $p$, we have
\begin{align} \label{eq:rt-rhot}
\sum_{t=1}^{T-1}\frac{r_t}{\rho_t} = \gO \left(\frac{1}{D}\sum_{t=1}^{T-1}(t+q)^{3/2}r_t \right)= \gO \left(\frac{1}{D}\left(\sum_{t=1}^{T-1}(t+q)^2r_t^2\right)^{1/2}
\left(\sum_{t=1}^{T-1}(t+q)\right)^{1/2} \right) = \gO(T).
\end{align}
For the total costs contributed by ${\rm (II)}$, the setting of $\delta_T=D/T^{3/2}$ implies
\begin{align*}
\sum_{t=0}^{T-1}\log_+\frac{\rho_t}{\delta_T} \le c_p''T+\frac32\sum_{t=0}^{T-1}\log\frac{T}{t+1} =c_p''T+\frac32\bigl(T\log T-\log(T!)\bigr)=\gO(T),
\end{align*}
where the final inequality uses $\log(T!)\ge T\log T-T$ and $c_p''>0$ is a numerical constant that only depends on $p$. Therefore, the claim follows from Theorem \ref{thm:outer-loop}.
\end{proof}


\section{Conclusion and Future Work}

In this paper, we propose a novel inexact accelerated Halpern iteration that solves smooth monotone variational inequalities with a $p$th-order oracle complexity of $\gO(\eps^{-2/(3p-1)})$. Our result is indeed optimal as it matches the lower bound given in \citet{chen2026solving}. In the future, it is interesting to study the optimal complexity under more general settings, including stochastic problems \citep{hsieh2019convergence,chen2024near}, finite-sum problems \citep{alacaoglu2022stochastic,cai2024variance}, and inexact second-order methods \citep{agafonov2024exploring,chen2025second}.




\bibliographystyle{plainnat}
\bibliography{sample}

@article{lin2023monotone,
  title={Monotone inclusions, acceleration, and closed-loop control},
  author={Lin, Tianyi and Jordan, Michael I},
  journal={Mathematics of Operations Research},
  volume={48},
  number={4},
  pages={2353--2382},
  year={2023},
  publisher={INFORMS}
}

@inproceedings{hsieh2019convergence,
  title={On the convergence of single-call stochastic extra-gradient methods},
  author={Hsieh, Yu-Guan and Iutzeler, Franck and Malick, J{\'e}r{\^o}me and Mertikopoulos, Panayotis},
  booktitle={NeurIPS},
  year={2019}
}

@inproceedings{alacaoglu2022stochastic,
  title={Stochastic variance reduction for variational inequality methods},
  author={Alacaoglu, Ahmet and Malitsky, Yura},
  booktitle={COLT},
  year={2022}
}

@article{agafonov2024exploring,
  title={Exploring Jacobian inexactness in second-order methods for variational inequalities: Lower bounds, optimal algorithms and quasi-Newton approximations},
  author={Agafonov, Artem and Ostroukhov, Petr and Mozhaev, Roman and Yakovlev, Konstantin and Gorbunov, Eduard and Tak{\'a}{\v{c}}, Martin and Gasnikov, Alexander and Kamzolov, Dmitry},
  booktitle={NeurIPS},
  year={2024}
}

@article{zhang2026matching,
  title={Matching Higher-Order Oracle Complexity for Smooth Monotone Variational Inequalities},
  author={Zhang, Haihan and Wu, Wendao and Zhang, Chenheng and Fang, Cong and Li, Haoxuan and Lin, Zhouchen},
  journal={arXiv preprint arXiv:2609.13462},
  year={2026}
}

@article{chen2026halpern,
  title={Halpern Iteration Achieves $\mathcal{O}(\epsilon^{-1/p})$ $p$th-Order Oracle Complexity for Monotone Variational Inequalities},
  author={Chen, Lesi and Zhang, Xinliang and Wang, Hengyu and Liu, Chengchang and Chen, Yongchao and Zhang, Jingzhao},
  journal={arXiv preprint arXiv:2608.08463},
  year={2026}
}

@article{combettes2018monotone,
  title={Monotone operator theory in convex optimization},
  author={Combettes, Patrick L.},
  journal={Mathematical Programming},
  volume={170},
  number={1},
  pages={177--206},
  year={2018},
  publisher={Springer}
}

@article{ryu2016primer,
  title={Primer on monotone operator methods},
  author={Ryu, Ernest K. and Boyd, Stephen},
  journal={Appl. comput. math},
  volume={15},
  number={1},
  pages={3--43},
  year={2016}
}

@inproceedings{cai2024variance,
  title={Variance reduced halpern iteration for finite-sum monotone inclusions},
  author={Cai, Xufeng and Alacaoglu, Ahmet and Diakonikolas, Jelena},
  booktitle={ICLR},
  year={2024}
}

@article{halpern1967fixed,
  title={Fixed points of nonexpanding maps},
  author={Halpern, Benjamin},
  year={1967}
}

@article{lieder2020convergence,
  title={On the convergence rate of the Halpern-iteration},
  author={Lieder, Felix},
  journal={Optimization letters},
  volume={15},
  number={2},
  pages={405--418},
  year={2020},
  publisher={Berlin, Heidelberg: Springer}
}

@inproceedings{diakonikolas2020halpern,
  title={Halpern iteration for near-optimal and parameter-free monotone inclusion and strong solutions to variational inequalities},
  author={Diakonikolas, Jelena},
  booktitle={COLT},
  year={2020}
}

@inproceedings{cai2024accelerated,
  title={Accelerated Algorithms for Constrained Nonconvex-Noncancave Min-Max Optimization and Comonotone Inclusion},
  author={Yang Cai and Argyris Oikonomou and Weiqiang Zheng},
  year={2024},
  booktitle={ICML},
}

@inproceedings{cai2023accelerated,
  title={Accelerated Single-Call Methods for Constrained Min-Max Optimization},
  author={Yang Cai and Weiqiang Zheng},
  year={2023},
  booktitle={ICLR}
}

@article{goodfellow2020generative,
  title={Generative adversarial networks},
  author={Goodfellow, Ian and Pouget-Abadie, Jean and Mirza, Mehdi and Xu, Bing and Warde-Farley, David and Ozair, Sherjil and Courville, Aaron and Bengio, Yoshua},
  journal={Communications of the ACM},
  volume={63},
  number={11},
  pages={139--144},
  year={2020},
  publisher={ACM New York, NY, USA}
}

@book{cesa2006prediction,
  title={Prediction, learning, and games},
  author={Cesa-Bianchi, Nicolo and Lugosi, G{\'a}bor},
  year={2006},
  publisher={Cambridge university press}
}

@article{ba2025doubly,
  title={Doubly optimal no-regret online learning in strongly monotone games with bandit feedback},
  author={Ba, Wenjia and Lin, Tianyi and Zhang, Jiawei and Zhou, Zhengyuan},
  journal={Operations Research},
  volume={73},
  number={6},
  pages={3219--3244},
  year={2025},
  publisher={INFORMS}
}

@book{giannessi1995variational,
  title={Variational inequalities and network equilibrium problems},
  author={Giannessi, Franco and Maugeri, Antonino and others},
  year={1995},
  publisher={Springer}
}

@book{facchinei2003finite,
  title={Finite-dimensional variational inequalities and complementarity problems},
  author={Facchinei, Francisco and Pang, Jong-Shi},
  year={2003},
  publisher={Springer}
}

@article{jiang2021optimal,
  title={An optimal high-order tensor method for convex optimization},
  author={Jiang, Bo and Wang, Haoyue and Zhang, Shuzhong},
  journal={Mathematics of Operations Research},
  volume={46},
  number={4},
  pages={1390--1412},
  year={2021},
  publisher={INFORMS}
}

@article{ralph2000superlinear,
  title={Superlinear convergence of an interior-point method despite dependent constraints},
  author={Ralph, Daniel and Wright, Stephen J.},
  journal={Mathematics of Operations Research},
  volume={25},
  number={2},
  pages={179--194},
  year={2000},
  publisher={INFORMS}
}

@article{qi2002smoothing,
  title={Smoothing functions and smoothing Newton method for complementarity and variational inequality problems},
  author={Qi, Liqun and Sun, Defeng},
  journal={Journal of Optimization Theory and Applications},
  volume={113},
  pages={121--147},
  year={2002},
  publisher={Springer}
}

@inproceedings{garg2021near,
  title={Near-optimal lower bounds for convex optimization for all orders of smoothness},
  author={Garg, Ankit and Kothari, Robin and Netrapalli, Praneeth and Sherif, Suhail},
  booktitle={NeurIPS},
  year={2021}
}

@inproceedings{agarwal2018lower,
  title={Lower bounds for higher-order convex optimization},
  author={Agarwal, Naman and Hazan, Elad},
  booktitle={COLT},
  year={2018}
}

@article{nesterov2018lectures,
  title={Lectures on convex optimization},
  author={Nesterov, Yurii},
  volume={137},
  year={2018},
  publisher={Springer}
}

@article{nemirovskij1983problem,
  title={Problem complexity and method efficiency in optimization},
  author={Nemirovskij, Arkadij Semenovi{\v{c}} and Yudin, David Borisovich},
  year={1983},
  publisher={Wiley-Interscience}
}

@article{jiang2022generalized,
  title={Generalized optimistic methods for convex-concave saddle point problems},
  author={Jiang, Ruichen and Mokhtari, Aryan},
  journal={SIAM Journal on Optimization},
  volume={35},
  number={3},
  pages={2066--2097},
  year={2025},
  publisher={SIAM}
}

@inproceedings{chen2025solving,
  title={Solving Convex-Concave Problems with $\mathcal{O}(\epsilon^{-4/7})$
 Second-Order Oracle Complexity},
  author={Chen, Lesi and Liu, Chengchang and Luo, Luo and Zhang, Jingzhao},
  booktitle={COLT},
  year={2025}
}

@inproceedings{chen2025second,
  title={Second-order min-max optimization with lazy hessians},
  author={Chen, Lesi and Liu, Chengchang and Zhang, Jingzhao},
  booktitle={ICLR},
  year={2025}
}

@article{chen2026solving,
  title={Solving Convex-Concave Problems with $\mathcal{O}(\epsilon^{-4/(3p+1)})$ $p$th-Order Oracle Complexity},
  author={Chen, Lesi and Zhang, Xinliang and Liu, Chengchang and Li, Junru and Luo, Luo and Zhang, Jingzhao},
  journal={arXiv preprint arXiv:2604.19462},
  year={2026}
}

@inproceedings{ying2016stochastic,
  title={Stochastic online AUC maximization},
  author={Ying, Yiming and Wen, Longyin and Lyu, Siwei},
  booktitle={NeurIPS},
  year={2016}
}

@inproceedings{zhang2018mitigating,
  title={Mitigating unwanted biases with adversarial learning},
  author={Zhang, Brian Hu and Lemoine, Blake and Mitchell, Margaret},
  booktitle={Proceedings of the 2018 AAAI/ACM Conference on AI, Ethics, and Society},
  pages={335--340},
  year={2018}
}

@inproceedings{carmon2022distributionally,
  title={Distributionally robust optimization via ball oracle acceleration},
  author={Carmon, Yair and Hausler, Danielle},
  booktitle={NeurIPS},
  year={2022}
}

@article{monteiro2013accelerated,
  title={An accelerated hybrid proximal extragradient method for convex optimization and its implications to second-order methods},
  author={Monteiro, Renato DC and Svaiter, Benar Fux},
  journal={SIAM Journal on Optimization},
  volume={23},
  number={2},
  pages={1092--1125},
  year={2013},
  publisher={SIAM}
}

@article{monteiro2012iteration,
  title={Iteration-complexity of a Newton proximal extragradient method for monotone variational inequalities and inclusion problems},
  author={Monteiro, Renato DC and Svaiter, Benar Fux},
  journal={SIAM Journal on Optimization},
  volume={22},
  number={3},
  pages={914--935},
  year={2012},
  publisher={SIAM}
}

@article{lin2022perseus,
  title={Perseus: A simple high-order regularization method for variational inequalities},
  author={Lin, Tianyi and Jordan, Michael I.},
  journal={Mathematical Programming},
  pages={1--42},
  year={2024},
  publisher={Springer}
}

@article{adil2022optimal,
  title={Optimal methods for higher-order smooth monotone variational inequalities},
  author={Adil, Deeksha and Bullins, Brian and Jambulapati, Arun and Sachdeva, Sushant},
  journal={arXiv preprint arXiv:2205.06167},
  year={2022}
}

@inproceedings{jiang2020improved,
  title={An improved cutting plane method for convex optimization, convex-concave games, and its applications},
  author={Jiang, Haotian and Lee, Yin Tat and Song, Zhao and Wong, Sam Chiu-wai},
  booktitle={SIGACT},
  year={2020}
}

@inproceedings{nesterov1983method,
  title={A method for solving the convex programming problem with convergence rate $\mathcal{O} (1/k^2)$},
  author={Nesterov, Yurii},
  booktitle={Dokl akad nauk Sssr},
  volume={269},
  pages={543},
  year={1983}
}

@article{korpelevich1976extragradient,
  title={The extragradient method for finding saddle points and other problems},
  author={Korpelevich, Galina M},
  journal={Matecon},
  volume={12},
  pages={747--756},
  year={1976}
}

@article{rockafellar1976monotone,
  title={Monotone operators and the proximal point algorithm},
  author={Rockafellar, R. Tyrrell},
  journal={SIAM Journal on Control and Optimization},
  volume={14},
  number={5},
  pages={877--898},
  year={1976},
  publisher={SIAM}
}

@article{nesterov2007dual,
  title={Dual extrapolation and its applications to solving variational inequalities and related problems},
  author={Nesterov, Yurii},
  journal={Mathematical Programming},
  volume={109},
  number={2-3},
  pages={319--344},
  year={2007},
  publisher={Springer}
}

@article{nemirovski2004prox,
  title={Prox-method with rate of convergence $\mathcal{O}(1/t)$ for variational inequalities with Lipschitz continuous monotone operators and smooth convex-concave saddle point problems},
  author={Nemirovski, Arkadi},
  journal={SIAM Journal on Optimization},
  volume={15},
  number={1},
  pages={229--251},
  year={2004},
  publisher={SIAM}
}

@article{arjevani2019oracle,
  title={Oracle complexity of second-order methods for smooth convex optimization},
  author={Arjevani, Yossi and Shamir, Ohad and Shiff, Ron},
  journal={Mathematical Programming},
  volume={178},
  number={1},
  pages={327--360},
  year={2019},
  publisher={Springer}
}

@article{nocedal1999numerical,
  title={Numerical optimization},
  author={Nocedal, Jorge and Wright, Stephen J.},
  year={1999},
  publisher={Springer}
}

@article{kinderlehrer2000introduction,
  title={An introduction to variational inequalities and their applications},
  author={Kinderlehrer, David and Stampacchia, Guido},
  year={2000},
  publisher={SIAM}
}

@article{alves2023search,
  title={A search-free $\mathcal{O}(1/k^{3/2}) $ homotopy inexact proximal-Newton extragradient algorithm for monotone variational inequalities},
  author={Alves, M. Marques and Svaiter, Benar F.},
  journal={SIAM Journal on Optimization},
  volume={34},
  number={4},
  pages={3235--3258},
  year={2024},
  publisher={SIAM}
}

@article{hartman1966some,
  title={On some non-linear elliptic differential-functional equations},
  author={Hartman, Philip and Stampacchia, Guido},
  year={1966}
}

@article{minty1962monotone,
  title={Monotone (nonlinear) operators in Hilbert space},
  author={Minty, George J},
  year={1962}
}

@inproceedings{jiang2024adaptive,
  title={Adaptive and Optimal Second-order Optimistic Methods for Minimax Optimization},
  author={Jiang, Ruichen and Kavis, Ali and Jin, Qiujiang and Sanghavi, Sujay and Mokhtari, Aryan},
  booktitle={NeurIPS},
  year={2024}
}

@article{bullins2022higher,
  title={Higher-order methods for convex-concave min-max optimization and monotone variational inequalities},
  author={Bullins, Brian and Lai, Kevin A.},
  journal={SIAM Journal on Optimization},
  volume={32},
  number={3},
  pages={2208--2229},
  year={2022},
  publisher={SIAM}
}

@article{chen2024near,
  title={Near-optimal algorithms for making the gradient small in stochastic minimax optimization},
  author={Chen, Lesi and Luo, Luo},
  journal={JMLR},
  volume={25},
  number={387},
  pages={1--44},
  year={2024}
}

@inproceedings{cai2022finite,
  title={Finite-time last-iterate convergence for learning in multi-player games},
  author={Cai, Yang and Oikonomou, Argyris and Zheng, Weiqiang},
  booktitle={NeurIPS},
  year={2022}
}

@article{nesterov2006cubic,
  title={Cubic regularization of Newton method and its global performance},
  author={Nesterov, Yurii and Polyak, Boris T},
  journal={Mathematical Programming},
  volume={108},
  number={1},
  pages={177--205},
  year={2006},
  publisher={Springer}
}

@inproceedings{carmon2022optimal,
  title={Optimal and adaptive monteiro-svaiter acceleration},
  author={Carmon, Yair and Hausler, Danielle and Jambulapati, Arun and Jin, Yujia and Sidford, Aaron},
  booktitle={NeurIPS},
  year={2022}
}

@article{jordan2025adaptive,
  title={Adaptive, doubly optimal no-regret learning in strongly monotone and exp-concave games with gradient feedback},
  author={Jordan, Michael and Lin, Tianyi and Zhou, Zhengyuan},
  journal={Operations Research},
  volume={73},
  number={3},
  pages={1675--1702},
  year={2025},
  publisher={INFORMS}
}

@inproceedings{kovalev2022first,
  title={The first optimal acceleration of high-order methods in smooth convex optimization},
  author={Kovalev, Dmitry and Gasnikov, Alexander},
  booktitle={NeurIPS},
  year={2022}
}

@article{huang2022approximation,
  title={An approximation-based regularized extra-gradient method for monotone variational inequalities},
  author={Huang, Kevin and Zhang, Shuzhong},
journal = {SIAM Journal on Optimization},
volume = {35},
number = {3},
pages = {1469-1497},
year = {2025}
}

@article{nesterov2023high,
  title={High-Order Reduced-Gradient Methods for Composite Variational Inequalities},
  author={Nesterov, Yurii},
  journal={arXiv preprint arXiv:2311.15154},
  year={2023}
}

@article{nesterov2008accelerating,
  title={Accelerating the cubic regularization of Newton’s method on convex problems},
  author={Nesterov, Yurii},
  journal={Mathematical Programming},
  volume={112},
  number={1},
  pages={159--181},
  year={2008},
  publisher={Springer}
}

@article{demmel2007fast,
  title={Fast linear algebra is stable},
  author={Demmel, James and Dumitriu, Ioana and Holtz, Olga},
  journal={Numerische Mathematik},
  volume={108},
  number={1},
  pages={59--91},
  year={2007},
  publisher={Springer}
}

@inproceedings{duan2023faster,
  title={Faster matrix multiplication via asymmetric hashing},
  author={Duan, Ran and Wu, Hongxun and Zhou, Renfei},
  booktitle={FOCS},
  year={2023}
}

@inproceedings{gasnikov2019optimal,
  title={Optimal tensor methods in smooth convex and uniformly convexoptimization},
  author={Gasnikov, Alexander and Dvurechensky, Pavel and Gorbunov, Eduard and Vorontsova, Evgeniya and Selikhanovych, Daniil and Uribe, C{\'e}sar A},
  booktitle={COLT},
  year={2019}
}

@inproceedings{bubeck2019near,
  title={Near-optimal method for highly smooth convex optimization},
  author={Bubeck, S{\'e}bastien and Jiang, Qijia and Lee, Yin Tat and Li, Yuanzhi and Sidford, Aaron},
  booktitle={COLT},
  year={2019}
}

\end{document}